\documentclass[12pt]{amsart}
\usepackage{epsfig,color}
\usepackage{amssymb} 
\usepackage{caption} 

\makeatother

\theoremstyle{definition}

\def\fnum{equation}
\newtheorem{Thm}[\fnum]{Theorem}
\newtheorem{Cor}[\fnum]{Corollary}

\newtheorem{Lem}[\fnum]{Lemma}

\newtheorem{Pro}[\fnum]{Proposition}

\numberwithin{equation}{section}
\newcommand{\Vol}{{\text{Vol}}}

\newcommand{\Hess}{{\text {Hess}}}

\def\RR{{\bold R}}

\def\SS{{\bold S}}

\newcommand{\e}{{\text {e}}}

\newcommand{\cW}{{\mathcal{W}}}

\newcommand{\cS}{{\mathcal{S}}}

\newcommand{\eqr}[1]{(\ref{#1})}

\title{Deficit functions and the log Sobolev inequality}
\author[]{Tobias Holck Colding}%
\address{MIT, Dept. of Math.\\
77 Massachusetts Avenue, Cambridge, MA 02139-4307.}
\author[]{William P. Minicozzi II}%
 
\thanks{The  authors
were partially supported by NSF  DMS Grants   2405393 and 2304684.}

\email{colding@math.mit.edu  and minicozz@math.mit.edu}

\begin{document}

\maketitle

\begin{abstract}
There is a long history of parabolic monotonicity formulas that developed independently from several different fields and a much more recent
elliptic theory.  The elliptic theory can be localized and there are additional monotone quantities.  There is also a surprising link: Taking a high-dimensional limit of the right elliptic monotonicity can give a parabolic one as a limit.   
Poincar\'e was the first to observe such a connection.    We introduce two deficit functions, one elliptic and one parabolic, then show that the parabolic deficit is 
pointwise the limit of the elliptic  
and, that the elliptic satisfies an equation that converges   to the equation for the parabolic.
These pointwise quantities and their equations recover the monotonicities and leads to an elliptic proof of the log Sobolev inequality as well as new concentration of measure phenomena.
\end{abstract}

\section{Introduction}

There is a rich theory of entropy monotonicity for the heat equation that links PDE, probability, statistical mechanics, information theory, and a number of sharp functional inequalities, see, e.g., \cite{Bj, Le3, V}.
The link comes from regarding a positive solution $u$ of the heat equation as a probability measure and studying its evolution over time.  There are  a number of monotonicity formulas, starting with the  Shannon-Boltzmann entropy and its derivative, the Fisher information.
If one combines monotonicity with the fact that $u$ looks more and more like a scaled gaussian (i.e., the central limit theorem), then one deduces sharp functional inequalities
such as the log Sobolev inequality.
 
There is a parallel theory of elliptic monotonicities  discovered much more recently, starting with \cite{C}; see \cite{CM1,CM2,AFM,AMO,AMMO}.  These elliptic monotonicities can be localized in a way that
the parabolic cannot, giving rich additional structure.   
Moreover, there is often a way to obtain parabolic results as a limit of  elliptic ones on high dimensional spaces.  The first example of this was the following observation of Poincar\'e, \cite{Po}, from 1912; cf. \cite{Bo, Mi}:  
\vskip1mm
\begin{quote}
A gaussian measure on $\RR^n$ may be viewed as a limit, where $N\to\infty$, of measures induced by orthogonal projection from $\RR^N$ to $\RR^n$ of the  uniform probability measure on $\sqrt{2N}\,\SS^{N-1}$.  
\end{quote}
\vskip1mm
The functional version of Poincar\'e's observation is that integrating a function $f(x)$ on $\RR^n$ with respect to the gaussian measure is asymptotically the same as integrating $f(x,y)=f(x)$ on $(x,y)\in \sqrt{2N}\,\SS^{N-1}$ with respect to the uniform probability measure.  
%Because of concentration of measure in high dimension, this is asymptotically the same as integrating $f(x,y)$ on $B_{\sqrt{N}}$ with respect to   Lebesgue measure.
Motivated by Poincar\'e's idea, McKean, \cite{M}, interpreted Wiener measure as the uniform distribution on an infinite  dimensional spherical surface of infinite radius.     Borell, \cite{Bo},   used Poincar\'e's projection of the spherical measure from $\sqrt{2N}\,\SS^{N-1}$.   By letting $N\to \infty$, Borell{\footnote{Our normalization differs from the one that Borell uses because we use a gaussian of a different scale.}}
 obtained the gaussian Brunn-Minkowski inequality on $\RR^n$ from a spherical Brunn-Minkowski inequality, \cite{Sc}, on 
$\sqrt{2N}\,\SS^{N-1}$.  Another example was given by Perelman who showed that his monotonicity of  reduced volume was the limit of the Bishop-Gromov volume comparison; see \cite{P}, cf. \cite{T, D}.
In \cite{BR}, Bustamante and Reiris showed that Perelman's entropy monotonicity, \cite{P}, for the Ricci flow was a  limit of  
Colding's monotonicities for harmonic functions, \cite{C}.  We will further develop this perspective here, leading to new results and giving a new understanding of some classical 
results, including   the log Sobolev inequality, \cite{Gr}.

The reason that the parabolic quantities involve gaussian integrals, while the elliptic ones are typically polynomial, comes from concentration of measure phenomena on high dimensional spheres and the gaussian decay that occurs away from the equator, cf. \cite{Le1, Le2, Mi, N, van}.
%\footnote{The idea of concentration of measures was put forward in the 1970s by V.D. Milman in the local theory of Banach spaces, building on work of Paul Levy.}
There is a related feature that comes up here, where we must understand the volumes of intersections of large dimensional spheres.  This leads to new concentration phenomena that play a key role in our arguments.

\subsection{Deficit functions}

We will define   functions $D_0$ and $d_0$, that we  call deficit functions,  
that capture the behavior of the heat equation and Laplace equation.  Both  functions   satisfy natural differential equations, 
leading to estimates and monotonicity formulas.
Theorem \ref{t:010} will show that the high-dimensional limit of the elliptic deficit $d_0$ gives the parabolic deficit $D_0$ and  the elliptic equation for $d_0$ 
converges to the parabolic equation for $D_0$.  

Let $M^n$ be an $n$-dimensional manifold.  For a positive function $u(x,t)$ on $M\times (0,\infty)$ define $f=f(x,t)$ and $D_0=D_0(x,t)$ by
\begin{align} \label{e:defu1}
u&=t^{-\frac{n}{2}}\,\exp\,\left(-f\right)\,   ,\\ 
D_0&=t\,\left(|\nabla_x f|^2+f_t\right)+(t\,f)_t\,   . \label{e:newDH1}
\end{align}
When $M$ is a closed $n$-dimensional manifold with non-negative Ricci curvature, the Li-Yau differential Harnack inequality, \cite{LY}, for a positive solution $u$ to the heat equation $\Box\,u=0$ (where $\Box\,u=(\partial_t-\Delta)\,u$) is the sharp inequality that 
\begin{align}
\Delta\,f - \frac{n}{2t} =|\nabla_x f|^2+f_t\leq  0 \,  .
\end{align}
The key in the proof  is that $t\,\Delta\,f$ is a sub-solution to an in-homogenous drift heat equation.  

  We will call $D_0$  {\it{the parabolic deficit function}} for reasons that will become clear.   Even though it is not part of the Li-Yau argument, it is a very natural quantity.  First, it is a sub-solution to a homogeneous drift heat equation.   
    For a general solution $u$,  
\begin{align}   \label{e:boxD0}
\Box\,(D_0 \,u)&=-2\,t\,\left| \Hess_f-\frac{\delta_{ij}}{2\,t} \right|^2\,u\,   .
\end{align}
  Second, it is constant on the Euclidean heat kernel  $u(x,t)=(4\, \pi \, t)^{-\frac{n}{2}}\,\e^{-\frac{|x|^2}{4\,t}}$, 
  with $f=\frac{|x|^2}{4\,t}+ \frac{n}{2} \, \log 4\pi$ and $D_0\equiv  \frac{n}{2} \, \log 4\pi$. 
  When  $\int u=1$, the integral of $D_0$ is related to the $\cW$ functional{\footnote{The $\cW$ functional is equal to $t\int \frac{|\nabla u|^2}{u}-\int u\,\log u - \frac{n}{2} \, \log t$, \cite{P}.}}
 by
\begin{align}
	\int D_0 \, u = \cW - n  
	\, .\label{e:logsobD0}
\end{align}
When one integrates $D_0$ against the weight $u$, the part of the integral that comes from $t\,f_t\,u$ contributes in a trivial way, whereas the part that comes from integrating $f\,u$ contributes with the Shannon-Boltzmann entropy.  However,  both are crucial to include.

To explain how the parabolic situation follows from the elliptic, let $M^n$ be an $n$-dimensional manifold and $\RR^N$ a large dimensional Euclidean space (for simplicity, assume that $M=\RR^n$, though it works much more generally).  We will use coordinates $x\in M$ and $y\in \RR^N$.  For a  solution $u>0$ to the heat equation on $M\times (0,\infty)$ define $v$ and $b=b(x,y)$ on $M\times \RR^N$ by
\begin{align}
v(x,y)&=r^{2-m}\,\exp\,(-f(x,\tau))  \label{e:othervxy}
\,  ,\\
b&=v^{\frac{1}{2-m}}\,  , \label{e:defu3}
\end{align}
where $r=|y|$ and $\tau=\frac{r^2}{2\,N}$.  
When $u$ satisfies the heat equation on $M$, then $v$ is harmonic on $M\times \RR^N$ up to higher order terms (as was observed by Perelman, \cite{P}).  
 The key link is what we call {\it{the (elliptic) deficit function}} $d_0$ defined by
\begin{align}	\label{e:defd0}
d_0(x,y) &= 2\, m \, [|\nabla b|^2-1]\,  .
\end{align}
If $M$ has nonnegative Ricci curvature and $w>0$ is a harmonic function appropriately normalized, then the gradient estimate in \cite{C} show that $|\nabla b|^2-1\leq 0$, where $b=w^{\frac{1}{2-n}}$, with equality if and only if $w$ is the Green's function and $M$ Euclidean space.  
This suggests that $d_0$ should be $\leq 0$ (at least asymptotically) with equality only when $M$ is Euclidean space and $v$ is the Green's function.  It is in this sense that $d_0$ is the deficit.      

Equation \eqr{e:key2} below relates the elliptic and parabolic deficit functions.  Equation \eqr{e:key1} shows that the elliptic deficit satisfies an elliptic equation and \eqr{e:key3} that  asymptotically this is the one of the parabolic deficit.   In this theorem, $B = \Hess_{b^2}-\frac{\Delta\,b^2}{m}\,\delta_{ij}$. 

\begin{Thm} 	\label{t:010}
If  $\frac{|y|^2}{2\, N} = \tau \in [\tau_1 , \tau_2]$, then 
\begin{align}
  \frac{\Delta\,(d_0\,v)}{v} - \frac{m}{b^2} \, |B|^2 &=O(N^{-1}) \,  ,\label{e:key1}\\
 d_0 -4\,D_0 &=O(N^{-1})\,  ,\label{e:key2}\\
 \frac{\Delta\,(d_0\,v)}{v} +  4\, \frac{\Box\,(D_0\,u)}{u}    &=O(N^{-1})\,  ,\label{e:key3}
\end{align}
where $d_0$ and $v$ are evaluated at $(x,y)$, while $D_0$ and $ u$ are at $(x , \tau)$.
\end{Thm}

Monotonicity of $\cW$ for the heat equation can be deduced from \eqr{e:boxD0} together with  \eqr{e:logsobD0}, but  can also be seen from \eqr{e:key1}, \eqr{e:key2} by taking the limit as the dimension $N\to \infty$.  The logarithmic Sobolev inequality  can be deduced from \eqr{e:boxD0}, \eqr{e:logsobD0} and the central limit theorem, but it can also be seen from \eqr{e:key1}, \eqr{e:key2} together with concentration of measures in high dimensions.  Equation \eqr{e:key3} links asymptotically the infinitesimal versions of the two monotonicities.

Bustamante and Reiris, \cite{BR}, showed that the parabolic monotonicity followed from the elliptic.  Their result was, in a sense, an integrated version of what we show here.  In their argument,
 the manifolds need to be closed.  In contrast, our arguments are local.  We show that the quantities here satisfy differential equations and we can therefore allow non-compact manifolds like Euclidean space. In fact, the formulas we show here reveal a relationship between the two deficit functions and their equations.
 
\subsection{The log Sobolev inequality as an elliptic limit}
 The log Sobolev inequality, \cite{Gr}; cf. \cite{Fe, St},   is  a powerful functional inequality that is equivalent to a lower bound for $\int u \, D_0$.  Since 
$d_0 \to 4\, D_0$, there is a surprising link between the elliptic monotonicities and the log Sobolev.  In particular, the log Sobolev would follow from an asymptotic lower bound for the elliptic deficit function
\begin{align}	\label{e:eLS1}
	d_0 \geq 2n \, \log 4\pi + O(N^{-1}) \, .
\end{align}
This is essentially what we will prove, though in a weak sense.  
The deficit function $d_0$ involves one derivative, but there is also a lower regularity $C^0$ elliptic approach using the ``almost harmonic'' function $v$.  
Namely, the log Sobolev is a consequence (see Proposition \ref{p:equiv})
of showing that if
$R^2 = 2 \, N \, \tau$, $|\bar{y}| = R$ and $\bar{z} = (0,\bar{y})$, then
\begin{align}	\label{e:elliptLS}
	   \lim_{N \to \infty} \, R^{m-2} \, v(\bar{z})  
	=  \left( 4 \, \pi \right)^{ - \frac{n}{2}}  \, \int u(x,0) \, \e^{ - \frac{|x|^2}{4 \, \tau}} \, dx  
	 \, .
\end{align}
We will prove \eqr{e:elliptLS}, and thus the log Sobolev, using the function $v$ and taking the  limit as the dimension goes to infinity.  

\begin{Thm}	\label{p:repnf}
Equation  \eqr{e:elliptLS} holds and, from this, the log Sobolev inequality follows.
\end{Thm}

The elliptic setting is richer than the parabolic in the sense that   high dimensional limits of elliptic monotonicities yield a single parabolic monotonicity, but there is a variety of different elliptic ones and the elliptic ones can be localized in a way that disappears in the limit.
The log Sobolev is an illustration of this.

\section{Differential equation for the deficit}

%Let $M^n$ be an $n$-dimensional manifold and $\RR^N$ a large dimensional Euclidean space.  We will use coordinates $x\in M$ and $y\in \RR^N$.  
 The next theorem  gives the differential equation for the elliptic deficit function $d_0$ defined in \eqr{e:defd0}, establishing \eqr{e:key1} in Theorem \ref{t:010}.

\begin{Thm}   \label{p:ld}
If  $\frac{|y|^2}{2\, N} = \tau \in [\tau_1 , \tau_2]$, then 
\begin{align}
	\left| \frac{\Delta\,(d_0\,v)}{v} - \frac{m}{b^2} \, |B|^2 \right| \leq \frac{C_2}{N\, \sqrt{\tau_1}} + \frac{C_2}{N\, \tau_1} + \frac{C_2\, |d_0|}{N\, \tau_1}  \,  ,
\end{align}
where $C_2$ depends on $\tau_2 , f$ and $n$.
\end{Thm}

\subsection{Adding dimensions}

We begin with Perelman's observation that
$v$  is harmonic on the space $M\times \RR^N$ up to higher order terms (see page 13 in \cite{P} and proposition 4.2 in \cite{BR}):

\begin{Lem}	\label{l:expandDeltav}
If $\Box \, u = 0$,  $v$ is given by \eqr{e:othervxy}, 
and we set $\tau = \frac{|y|^2}{2\, N}$, then  
\begin{align}	
	\frac{\Delta \, v}{v} &= \frac{4\, u_t}{N\, u} + \frac{2\, \tau \, u_{tt}}{N\, u}
	 =   \frac{(n-2)\, n }{2\, N \, \tau}      +
	2\, (n - 2) \,  \frac{f_t}{N}
	+ \frac{(f_t^2 - f_{tt}) \, 2 \, \tau }{N}     \, . 
\end{align}
\end{Lem}

\begin{proof}
We write $\Delta = \Delta_x + \Delta_y$. First, we note that
 $
	\frac{\Delta_x v }{v}  = \frac{\Delta u}{u} = \frac{u_t}{u}  
$,
where the second equality used the heat equation.  Observe that 
\begin{align}
	\nabla_y \left[ u (x,|y|^2/2N) \right] = u_t \, \frac{y}{N} \, , 
\end{align}
so that $\Delta_y \left[ u (x,|y|^2/2N) \right] = u_t + u_{tt} \frac{|y|^2}{N^2}$.
Therefore, since $|y|^{2-N}$ is harmonic (away from $y=0$), we have that
\begin{align}
	(2\, N)^{ \frac{n}{2}} \, \Delta_y v &= 2\, \langle \nabla_y |y|^{2-N} , u_t \, \frac{y}{N} \rangle + |y|^{2-N} \, 
	\left(  u_t + u_{tt} \frac{|y|^2}{N^2} \right) = \left[ (4/N - 1) \, u_t  + u_{tt} \frac{|y|^2}{N^2}
	\right] \, |y|^{2-N} \, . \notag
\end{align}
Putting this together gives
\begin{align}
	\frac{\Delta v}{v} = \frac{4}{N} \, \frac{u_t}{u} + \frac{2\, \tau \, u_{tt}}{N\, u} \, ,
\end{align}
where $u$ and its derivatives are evaluated at $(x, |y|^2/2N)$.  This gives the first equality. The second follows from the formula 
\eqr{e:defu1} for $u$ in terms of $f$.
\end{proof}

\subsection{Differential equation for the deficit}

We will use $\psi$ to denote $b^m \, \Delta v$; recall that $B$ is the trace-free hessian of $b^2$.     We will need the following:

\begin{Lem}   \label{l:dfct}
Set $d= 2\, m \, |\nabla b|^2$, then we have 
 \begin{align}
  \Delta\, d &+  2\, (2-m) \, \langle\nabla \log b ,\nabla d \rangle
  	=\frac{m}{b^2}\, \left| B \right|^2 + \frac{4\, m }{2-m}\, \, \langle \nabla \log b , \nabla \psi
 \rangle  \notag \\
 &\quad  +\frac{  2 \, d  \, \psi 
	}{(2-m)\, b^2 }
+\frac{   4\, \psi^2 
	}{(2-m)^2\, b^2 }
\, .
\end{align}
\end{Lem}

\begin{proof}
For any strictly monotone function $f:\RR\to \RR$, we have that
\begin{align}   \label{e:Lap1}
\Delta\,f(u)=f''(u)\,|\nabla u|^2+f'(u)\,\Delta\,u=\frac{f''}{(f')^2}\,|\nabla f(u)|^2+f'(u)\,\Delta\,u\, .
\end{align}
Setting $f(s)=s^{\frac{2}{2-m}}$,  it follows that
\begin{align}
\Delta\,b^2&=2\,m\,|\nabla b|^2+\frac{2}{2-m}\,b^m\,\Delta\,v=2\,m\,|\nabla b|^2+\frac{2}{2-m}\,\psi\, .\label{e:Lapb^2}
\end{align}
The Bochner formula gives
\begin{align}
2\,b^2\,\Delta\,|\nabla b|^2&+4\,\langle\nabla b^2,\nabla |\nabla b|^2\rangle+2\,|\nabla b|^2\,\Delta\,b^2=
2\,\Delta\,(b^2\,|\nabla b|^2)=\frac{1}{2}\,\Delta\,|\nabla b^2|^2\notag\\
&=\left|\Hess_{b^2}-\frac{\Delta\,b^2}{m}\,\delta_{ij}\right|^2+\frac{|\Delta\,b^2|^2}{m}+\langle\nabla\,\Delta\,b^2,\nabla b^2\rangle\,  .
\end{align}
Next recall that
\begin{align}   \label{e:eDeltab2}
\Delta\,b^2&=2\,m\,|\nabla b|^2+\frac{2}{2-m}\,b^n\,\Delta\,v =2\,m\,|\nabla b|^2+\frac{2}{2-m}\,\psi\,  , \\
\frac{ (\Delta \, b^2)^2}{m} &= \frac{\Delta \, b^2}{m} \, \left( 2\,m\,|\nabla b|^2+\frac{2}{2-m}\,\psi \right) = 2\, |\nabla b|^2 \, \Delta \, b^2 + \frac{2\,  \psi \, \Delta \, b^2}{m\, (2-m)}  \, .
\end{align}
Inserting this,  cancelling terms and multiplying by $m$ gives
\begin{align}
2\,m \,b^2\,\Delta\,|\nabla b|^2&+4\, m\, \langle\nabla b^2,\nabla |\nabla b|^2\rangle\notag\\
&=m\, \left| B \right|^2+\frac{2\,\Delta\,b^2}{(2-m)}\,\psi
 +m \, \langle \nabla b^2 , \nabla ( 2\,m\,|\nabla b|^2+\frac{2}{2-m}\,\psi
 )\rangle \,  .
\end{align}
Setting $d=2\,m\,|\nabla b|^2$ and simplifying this gives
\begin{align}
  b^2\,\Delta\, d &+ (2-m) \, \langle\nabla b^2,\nabla d \rangle
  	=m\, \left| B \right|^2+\frac{2\,\Delta\,b^2}{(2-m)}\,\psi
 + \frac{2\, m }{2-m}\, \, \langle \nabla b^2 , \nabla \psi
 \rangle \,  .
\end{align}
Multiplying by $\frac{1}{b^2}$ and using the formula for $\Delta\, b^2$ again gives
 \begin{align}
  \Delta\, d &+  2\, (2-m) \, \langle\nabla \log b ,\nabla d \rangle
  	=\frac{m}{b^2}\, \left| B \right|^2+\frac{  2 \, d + 4 \psi/(2-m)
	}{(2-m)\, b^2 }\,\psi
 + \frac{4\, m }{2-m}\, \, \langle \nabla \log b , \nabla \psi
 \rangle \,  . \notag 
\end{align}
\end{proof}

We  use  Lemma \ref{l:dfct} to prove the key differential inequality for the elliptic deficit function:

\begin{Cor}   \label{c:deficit1}
We have
 \begin{align}
  \frac{ \Delta\, (d_0\, v) }{v}  &
  	=\frac{m}{b^2}\, \left| B \right|^2 + \frac{4\, m }{2-m}\, \, \langle \nabla \log b , \nabla \psi
 \rangle   +\frac{  (4-m) \, d_0  \, \psi 
	}{(2-m)\, b^2 } +\frac{  4 \, m  \, \psi 
	}{(2-m)\, b^2 }
+\frac{   4\, \psi^2 
	}{(2-m)^2\, b^2 }   \, . \notag 
\end{align}
\end{Cor}

\begin{proof}
Since
\begin{align}
	\Delta \, (d \, v) &= v \, \Delta \, d + 2\, \langle \nabla d , \nabla u \rangle + d \, \Delta v  = 
	v \, \left(  \Delta \, d + 2\, \langle \nabla d , \nabla \log v \rangle + d \, b^{m-2} \, \Delta v \right) \notag \\
	&= v \, \left(  \Delta \, d + 2\, (2-m) \, \langle \nabla d , \nabla \log b \rangle + d \, \frac{\psi}{b^2} \right)
	\, ,
\end{align}
it follows from Lemma \ref{l:dfct} that
\begin{align}
 \frac{ \Delta\, (d\, v) }{v}  &
  	=\frac{m}{b^2}\, \left| B \right|^2 + \frac{4\, m }{2-m}\, \, \langle \nabla \log b , \nabla \psi
 \rangle   +\frac{  2 \, d  \, \psi 
	}{(2-m)\, b^2 }
+\frac{   4\, \psi^2 
	}{(2-m)^2\, b^2 } + d \, \frac{\psi}{b^2}
\, ,  
\end{align}
To see the claim, we use that
\begin{align}
		\Delta \, (d_0 \, v)  = 	\Delta \, (d \, v) - 2\, m \, \Delta \, v =  	\Delta \, (d \, v)  - 2\, m \, v \, \frac{\psi}{b^2} \, .
\end{align}
\end{proof}

We need to estimate $\psi=b^m\,\Delta\,v$, its derivatives, and the derivatives of $b$.   

\begin{Lem}	\label{l:bmdeltau}
If  $\frac{|y|^2}{2\, N} = \tau \in [\tau_1 , \tau_2]$, then 
\begin{align}
	|\psi | , \, |\nabla_x \psi| , |\nabla_y b| \leq  C   {\text{ and }}
	|\nabla_y \psi| , \, | \nabla_x b| \leq   C\, N^{ - \frac{1}{2}}  \, ,
\end{align}
where $C$ depends on $n, f$ and $\tau_2$.
\end{Lem}

\begin{proof}
Within the proof, we will use $C$ for a constant that depends on $n, f$ and $\tau_2$ and is allowed to change even within an equation line.
Define $\psi^1 , \psi^2 , \psi^3$ by
\begin{align}
	\psi^1 = \frac{b^2}{r^2} , \, \psi^2 = \frac{b^2 \, f_t}{N} {\text{ and }} \psi^3 = \frac{b^2\, r^2}{N^2} \, (f_t^2 - f_{tt}) \, , 
\end{align}
so that
\begin{align}
	\psi = b^m \, \Delta \, v = n(n-2) \, \psi^1 + 2(n-2) \, \psi^2 + \psi^3 \, .
\end{align}
We will use freely that $f, f_t , f_{tt}$ and all of their $x$ and $t$ derivatives are bounded by such a $C$. In particular, since 
%\begin{align}
	$b = |y| \, \e^{ \frac{ f(x, |y|^2/2N)}{m-2}}$,
%\end{align}
we see that $b \leq C \, |y|$, $|\nabla b| \leq C$ and 
\begin{align}
	\left| \nabla \frac{b}{|y|} \right| \leq \frac{C}{m-2} \, .
\end{align}
This gives the desired gradient bounds on $\psi^1$.  Moreover, 
it follows immediately that each $|\psi^i|  \leq C$ and 
\begin{align}
	|\nabla_x b | = \left| |y| \, \nabla_x \frac{b}{|y|} \right| \leq \frac{C}{\sqrt{N}} \, .
\end{align}

\vskip1mm
When we estimate the first derivatives of $\psi^2$ and $\psi^3$, we will use that $|\nabla b|$ is bounded and $|\nabla r| \leq 1$.  It is also useful that $r=|y|$ does not depend on $x$.
    We have that
\begin{align}
	|\nabla_x \, \psi^2 |&= \left| \frac{1}{N} \, f_t \, \nabla_x b^2 +  \frac{1}{N} \, b^2 \, \nabla_x f_t \right|  \leq C  \,  , \\
	|\nabla_x \, \psi^3| &= \left|  \frac{r^2}{N^2} \, \left( b^2  \, \nabla_x f_t +2\, b \, f_t \nabla_x b
	\right) \right| \leq C \, .
\end{align}

\vskip1mm
In the last  part,  define $\zeta (x,t) = (f_t^2 - f_{tt})$; we have that $\zeta$ and its $x$ and $t$ derivatives are bounded.  We will identify $\zeta$ with $\zeta (x, |y|^2/(2N))$ below, so that
$|\nabla_y \zeta |= |\zeta_t \, \frac{y}{N} | \leq C  \, N^{ - \frac{1}{2}}$. 
We have that
\begin{align}
	|\nabla_y \, \psi^2| &= \left|   \frac{1}{N} \, \left( 2\, b \, f_t \, \nabla_y b + b^2 \, f_{tt} \frac{y}{N} \right) \right| \leq C  \, N^{ - \frac{1}{2}} \, , \\
	\left| \nabla_y \, \psi^3 \right| &= \frac{1}{N^2} \, \left| 2 \, b \, r^2 \, \zeta  \, \nabla b + 2 \, r \, b^2 \, \nabla_y r \,\zeta  +
	b^2 \, r^2 \, \zeta_t \, \frac{y}{N} \right|  \leq C \, N^{ - \frac{1}{2}}  \, .
\end{align}
\end{proof}

\begin{proof}[Proof of Theorem \ref{p:ld}]
Corollary \ref{c:deficit1} gives that
\begin{align}
 \frac{ \Delta\, (d_0\, v) }{v}  - \frac{m}{b^2}\, \left| B \right|^2    &
  	=  \frac{1}{2-m} \, 
	\left( 4\, m  \, \langle \nabla \log b , \nabla \psi
 \rangle   +\frac{  (4-m) \, d_0  \, \psi 
	}{b^2 } +\frac{  4 \, m  \, \psi 
	}{b^2 }
+\frac{   4\, \psi^2 
	}{(2-m)\, b^2 } \right)  \, .  \ \notag
\end{align}
We will bound the terms on the right using that Lemma \ref{l:bmdeltau} gives that 
\begin{align}
	|\psi | , |\nabla_x \psi| , |\nabla_y b| , (\sqrt{N} \, |\nabla_y \psi| ) , (\sqrt{N} \, |\nabla_x b|) \leq c_2  \, , 
\end{align}
where $c_2$ depends on $\tau_2 , f$ and $n$.  We have that
\begin{align}
	\left|  \langle \nabla \log b , \nabla \psi
 \rangle  \right| \leq b^{-1} \, (|\nabla_x b| \, |\nabla_x \psi| + |\nabla_y b| \, |\nabla_y \psi|) \leq  \frac{2 \, c_2^2 }{b\, \sqrt{N}} \leq \frac{2 \, c_2^2 }{\sqrt{\tau_1} \, N} \, .
\end{align}
Next, we have
\begin{align}
	\left| \frac{d_0 \, \psi}{b^2} \right| &\leq \frac{c_2\, |d_0|}{b^2} \leq C\,  \frac{c_2\, |d_0|}{N \, \tau_1 } \, , \\
	\left| \frac{\psi}{b^2} \right| &\leq C\,  \frac{c_2}{N \, \tau_1 } \, , \\
	\left| \frac{ \psi^2}{(m-2)^2 \, b^2} \right| &\leq C\, \frac{c_2^2}{N^3 \, \tau_1} \, .
\end{align}
The proposition now follows.
\end{proof}

%The next lemma gives an asymptotic expansion of $b$ on $M\times \RR^N$ in terms of $f$ and the auxiliary dimension $N$.  

%\begin{Lem}    \label{l:336}
%Let $\tau=\frac{r^2_0}{2\,N}$ be fixed, we have that
%\begin{align}   
%b^2&=r^2+2\,\frac{r^2}{N}\,f(x,\tau)+O(N^{-1})=r^2+4\,\tau\,f(x,\tau)+O(N^{-1})\,  .\label{e:approxb2}
%\end{align}
%\end{Lem}

%\begin{proof} 
%Since $N$ is large and $\tau=\frac{r^2_0}{2\,N}$ is fixed we have by Taylor expansion that 
%\begin{align}
%b^2&= r^2 \, \e^{ \frac{2\, f(x,\tau)}{m-2}} = 
%r^2\,\left(1+\frac{2\, f(x,\tau))}{N}+O(N^{-2})\right)\,  .
%\end{align}
%The claim easily follows.  
%\end{proof}

The next result relates the two deficit functions, proving \eqr{e:key2} in Theorem \ref{t:010}.

\begin{Pro}    \label{c:336}
Given $\tau=\frac{r^2_0}{2\,N}$  fixed, we have that
\begin{align}   
d_0 &= 4\,D_0 (x,\tau)+O(N^{-1})\,  .\label{e:Nover2} 
\end{align}
\end{Pro}

\begin{proof}
Note that
\begin{align}
\nabla b&=\frac{y}{|y|}\,\left(1+\frac{f(x,\frac{r^2}{2\,N})}{N}\right)+r\,\frac{\nabla_x\,f}{N}+\frac{r\,f_t\,y}{N^2}+O(N^{-2})\notag\\
&=\frac{y}{|y|}\,\left(1+\frac{f(x,\frac{r^2}{2\,N})}{N}+\frac{r^2}{N^2}\,f_t\right)+r\,\frac{\nabla_x\,f}{N}+O(N^{-2})\,  .
\end{align}
Therefore, we see that
\begin{align}
 \frac{d_0}{2} &= N\,\left(|\nabla b|^2-1\right)=2\,f+\frac{r^2}{N}\,|\nabla_x\,f|^2+\frac{2\,r^2}{N}\,f_t+O(N^{-1})\,  .
\end{align}
\end{proof}

\section{The deficit functions satisfy the same equation asymptotically}

%upshot

The next theorem shows that the parabolic equation for $D_0$ is the high-dimensional limit of the elliptic equations for $d_0$, proving \eqr{e:key3} and completing the proof of Theorem \ref{t:010}.

 \begin{Thm}	\label{t:old315}
Given $0 < \tau_1 < \tau_2$, then on the set with $\frac{|y|^2}{2\, N} = \tau \in [\tau_1 , \tau_2]$, we have 
\begin{align}
	\left| \frac{\Delta\,(d_0\,v)}{v} +  4\, \frac{\Box\,(D_0\,u)}{u} \right| \leq  \frac{C_2}{N\, \sqrt{\tau_1}} + \frac{C_2}{N\, \tau_1} + \frac{C_2}{N}\,  ,
\end{align}
where $d_0$ and $v$ are evaluated at $(x,y)$ and $D_0, u$  at $(x , \tau)$.
\end{Thm}

We will break the proof up into several steps.  The next lemma computes the asymptotics of the trace-free hessian $B$ of $b^2$.

\begin{Lem}	\label{l:matrixasy}
If we set $\tau = \frac{|y|^2}{2N}$, then the hessian of $b^2$ satisfies
\begin{align}
	\Hess_{b^2} = \frac{2\, b^2}{N}  \,  \left( \Hess_f + \frac{\delta_{\alpha \beta}}{2 \, \tau} + O(N^{ - \frac{1}{2}}) \, Q_1 + O(N^{-1}) \, Q_2 \right) \, ,
\end{align}
where $\delta_{\alpha \beta}$ is the $N \times N$ identity, $Q_1$ is off-diagonal and $Q_2$ is in the two diagonal blocks.  It follows from this that the trace-free hessian $B$ satisfies
\begin{align}
	|B|^2 &= \frac{4\, b^4}{N^2} \, \left( \left| \Hess_f -  \frac{1}{2\, \tau}  \, \delta_{ij} \right|^2 + O(N^{-1}) \right) \, .
\end{align}
\end{Lem}

\begin{proof}
Let $x_i$ and $y_{\alpha}$ be coordinates on $\RR^n$ and $\RR^N$, respectively.  Since $r=|y|$ does not depend on $x_i$, we have that
\begin{align}
	(b^2)_i &= b^2 \, \frac{2\, f_i}{m-2} \, ,  \\
	(b^2)_{ij} &= b^2 \, \left(  \frac{2\, f_{ij}}{m-2} +  \frac{4\, f_{i}\, f_j}{(m-2)^2} \right)\, .
\end{align}
Next, using that $(r^2)_{\alpha} = 2 \, y_{\alpha}$, we see that
\begin{align}
	(b^2)_{\alpha} &= 2\, y_{\alpha} \, \e^{ \frac{2 \, f}{m-2} } + b^2 \, \frac{2\, f_t}{m-2}  \, \frac{y_{\alpha}}{N}  =  y_{\alpha}\, \left( \frac{2}{r^2}  + \frac{2\, f_t}{N(m-2)} \right)  \, b^2  \, , \notag \\
		(b^2)_{i\, \alpha} &=  2\, y_{\alpha} \, \e^{ \frac{2 \, f}{m-2} } \, \frac{2\, f_i}{m-2} + b^2 \, \left( \frac{2\, f_i}{m-2} \right) \, \frac{2\, f_t}{N(m-2)}  \, y_{\alpha}
		+ y_{\alpha}\, \left(   \frac{2\, f_{ti}}{N(m-2)} \right)  \, b^2 \, , \notag  \\
		(b^2)_{\alpha \beta} &= 2\, \delta_{\alpha \beta} \, \e^{ \frac{2 \, f}{m-2} } + 4\, \frac{y_{\alpha} \, y_{\beta}\, f_t}{N(m-2)} \,   \e^{ \frac{2 \, f}{m-2} }+
		 b^2 \, \frac{2\, f_t\, \delta_{\alpha \beta}}{N(m-2)}   +
		  b^2 \, \left( \frac{4\, f_t^2}{(m-2)^2}  +
		    \frac{2\, f_{tt}}{m-2} \right) \, \frac{y_{\alpha}\, y_{\beta}}{N^2} \, . \notag
\end{align}
Noting that $r^2$ and $b^2$ will be on the order of $N$, we see that
\begin{align}
		b^{-2} \, (b^2)_{ij} &=  \frac{2\, f_{ij}}{m-2} +  O(N^{-2}) \, ,  \\
		b^{-2} \, (b^2)_{i\, \alpha} &= O(N^{ - \frac{3}{2}})    \, , \\
		b^{-2} \, (b^2)_{\alpha \beta} &=  \frac{2\, \delta_{\alpha \beta}}{ r^2} +   O(N^{-2})  \, .
\end{align}
Thus, if we set $\tau = \frac{r^2}{2N}$, then we get the first claim.

Taking the trace  gives
 $\Delta \, b^2 = \frac{2\, b^2}{N}  \,  \left(\frac{N}{2 \, \tau}   + O(1)  \right)$,
so the trace-free part $B$ is  
\begin{align}
	\Hess_{b^2} - \frac{\Delta \, b^2}{m} \, g_{ij} &= \frac{2\, b^2}{N}  \,  \left( \Hess_f - \left( \frac{1}{2\, \tau} + O(N^{-1}) \right) \, (\delta_{ij} + \delta_{\alpha \beta} )+ \frac{\delta_{\alpha \beta}}{2 \, \tau} + O(N^{ - \frac{1}{2}}) \, Q_1 + O(N^{-1}) \, Q_2 \right) \notag \\
	&=  \frac{2\, b^2}{N}  \,  \left( \Hess_f -  \frac{1}{2\, \tau}  \, \delta_{ij}  + O(N^{ - \frac{1}{2}}) \, Q_1 + O(N^{-1}) \, Q_2 \right) \, .
\end{align}
It follows that the norm squared is
\begin{align}
	|B|^2 &= \frac{4\, b^4}{N^2} \, \left( \left| \Hess_f -  \frac{1}{2\, \tau}  \, \delta_{ij} \right|^2 + O(N^{-1}) \right) \, .
\end{align}
\end{proof}

\begin{proof}[Proof of Theorem \ref{t:old315}]
Proposition \ref{p:ld} gives that 
if  $\frac{|y|^2}{2\, N} = \tau \in [\tau_1 , \tau_2]$, then 
\begin{align}	\label{e:fromABOVE}
	\left| \frac{\Delta\,(d_0\,v)}{v} - \frac{m}{b^2} \, |B|^2 \right| \leq \frac{C_2}{N\, \sqrt{\tau_1}} + \frac{C_2}{N\, \tau_1} + \frac{C_2\, |d_0|}{N\, \tau_1}  \,  ,
\end{align}
where $B = \Hess_{b^2}-\frac{\Delta\,b^2}{m}\,\delta_{ij}$
and 
 $C_2$ depends on $\tau_2 , f$ and $n$.
 The second ingredient is Lemma \ref{l:matrixasy}, which gives that
\begin{align}	\label{e:matrixA}
	|B|^2 &= \frac{4\, b^4}{N^2} \, \left( \left| F \right|^2 + O(N^{-1}) \right) = 16 \, \tau^2 \, \left( |F|^2 + O(N^{-1}) \right) \, ,
\end{align}
where $B$ and $b$ are evaluated at $(x,y)$ while $F = \Hess_f-\frac{1}{2\,t}\,g_{ij}$ is evaluated at $(x, \tau)$ with $\tau = \frac{|y|^2}{2N}$.  Since $\frac{b}{r} - 1$ and $\frac{m}{N} - 1$ are both $O(N^{-1})$, it follows that
\begin{align}
	 \left| 8 \, \tau \, |F|^2   - \frac{m}{b^2} \, |B|^2 \right| &=    \left| \frac{4 \, r^2}{N} \, |F|^2   - \frac{m}{b^2} \,  |B|^2 \right| \leq  \left| \frac{4 \, b^2}{N} \, |F|^2   - \frac{m}{b^2} \,  |B|^2 \right| + O(N^{ - \frac{1}{2}}) \notag \\
	 &\leq \frac{m}{b^2} \, \left| \frac{4 \, b^4}{N^2} \, |F|^2   -   |B|^2 \right| + O(N^{ - \frac{1}{2}})  
	 \leq  \tau_1^{-1} \, O(N^{-1})  + O(N^{ - \frac{1}{2}}) 
	 \, ,
\end{align}
where we used that $|F|$ is bounded (since $\tau \leq \tau_2$) and the last inequality used \eqr{e:matrixA}.
Combining this with \eqr{e:fromABOVE}, we see that
\begin{align}	\label{e:2f20}
	\left| \frac{\Delta\,(d_0\,v)}{v} -8 \, \tau \, |F|^2 \right| &\leq   \frac{C_2}{N\, \sqrt{\tau_1}} + \frac{C_2}{N\, \tau_1} + \frac{C_2}{N} \, .
\end{align}
Using the Leibniz rule, we see that
\begin{align}	\label{e:old221}
\Box\,& D_0 = \Box\,[t\,(|\nabla f|^2+f_t)+(t\,f)_t] = \Box \, f + 2 \, t \, \Box \, f_t + t \, \Box \, |\nabla f|^2 + |\nabla f|^2 + 2 \, f_t  \,  .
\end{align}
Since $\Box \, f = - \frac{n}{2\, t} - |\nabla f|^2$, we get  that $\Box\,f_t=\frac{n}{2\,t^2}-2\,\langle\nabla f_t,\nabla f\rangle$ and
\begin{align}
\Box\,|\nabla f|^2&=-2\,\left|  \Hess_f  \right|^2 -2\,\langle \nabla f,\nabla |\nabla f|^2\rangle \label{e:eF1}
\, .
\end{align}
Using these in \eqr{e:old221} gives that
\begin{align}	 
\Box\, D_0 = - \frac{n}{2\, t} - |\nabla f|^2 + 2 \, t \,  \left( \frac{n}{2\,t^2}-2\,\langle\nabla f_t,\nabla f\rangle \right)
-2\, t \,  \left( \left|  \Hess_f  \right|^2   +\langle \nabla f,\nabla |\nabla f|^2\rangle
\right) + |\nabla f|^2 + 2 \, f_t \notag \\
=  \frac{n}{2\, t}  -4 \, t \,\langle\nabla f_t,\nabla f\rangle 
-2\, t \,   \left|  \Hess_f  \right|^2   - 2\, t\, \langle \nabla f,\nabla |\nabla f|^2\rangle  + 2 \, f_t 
  \,  .
\end{align}
Using that $|F|^2 = |\Hess_f - \delta_{ij}/2t|^2 = |\Hess_f|^2 - \Delta f/t + n/4t^2$, we see that
\begin{align}	\label{e:old221a}
\Box\, D_0 
&=     -2\,\langle\nabla D_0,\nabla f\rangle -2\,t\,\left| F \right|^2
  \,  .
\end{align}
The product gives that
 $
	\frac{\Box\,(D_0 \,u)}{u}  =-2\,t\,\left| F \right|^2$ and comparing with
	  \eqr{e:2f20} completes the proof.
\end{proof}

\section{The log Sobolev inequality as a high dimensional elliptic limit}

The log Sobolev inequality is a classical result   closely linked to the heat equation.
 We will explain  how it
can be obtained from a purely elliptic  approach, proving Theorem \ref{p:repnf}.   This elliptic approach uses the previous sections
and a sharp asymptotic lower bound for the elliptic deficit function.  This in turn follows from an asymptotic formula \eqr{e:elliptLS} for
     the solid averages of the almost-harmonic function $v(x,y)$ on $\RR^m = \RR^n \times \RR^N$ defined in 
    \eqr{e:othervxy}.   These averages are on balls centered away from the origin and computing them   involves a concentration of measure phenomenon,  with a twist since the balls are off-center.

\vskip1mm
 We first show that Theorem \ref{p:repnf} follows if  the averages of $v$ 
  over  balls centered at  $\bar{z}= (0,\bar{y}) \in \RR^m$  with $|\bar{y}|^2 = 2\, N \, \tau$
 converge  to a gaussian integral of $u(x,0)$ on $\RR^n$:

\begin{Lem}	\label{l:setitup}
 Theorem \ref{p:repnf} will follow if we show that
\begin{align}	\label{e:IFWT}
	\lim_{\beta \to 1} \left\{\lim_{N \to \infty} \,    \frac{\beta^{-m} }{2\, \tau \, \omega_{m-1}}\, \int_{B_{\beta \, R}(\bar{z})} v(x,y) \, dy \, dx \right\}
	=  \left( 4 \, \pi \right)^{ - \frac{n}{2}}  \, \int_{\RR^n} u(\bar{x},0) \, \e^{ - \frac{|x|^2}{4 \, \tau}} \, d\bar{x} \, .
\end{align}
\end{Lem}

\vskip1mm
To  evaluate the integral on the left in \eqr{e:IFWT}, we first   fix $x$, then slice the $y$ integral by level sets $|y| =r$, and then  integrate over $x$.  It would be  natural to slice by the distance to $\bar{y}$, but $v(x,y)$ has a rotational symmetry in $y$  that forces us to slice the other way.

  \begin{figure}[htbp]
\includegraphics[totalheight=.25\textheight, width=.42\textwidth]{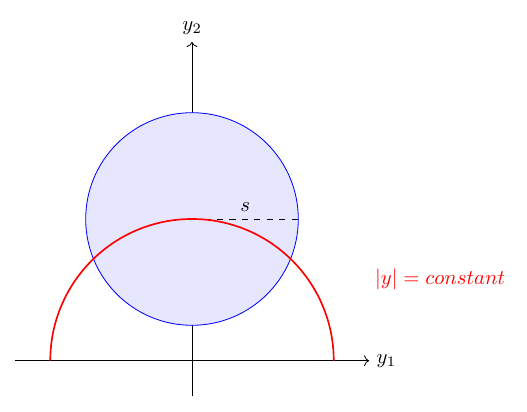}
 \caption*{For each fixed $x$, we integrate over all $y$ with $|y - \bar{y}|^2 \leq s^2 = (\beta \, R)^2 - |x|^2$ (this is the shaded region).  To evaluate the $y$-integral, we slice by $|y| =r$ constant.}
\end{figure}
 
   The  $|y| =r$  slices are   intersections of an $N$-dimensional ball with  spheres centered outside of the ball.  The volumes of the  slices concentrate sharply near certain radii, leading to
   gaussian integrals in the limit.  This is a concentration of measure phenomenon, but with a twist.  The usual concentration of measure is when a high-dimensional
   sphere is sliced by latitudes, i.e., intersected with spheres centered at  a pole, and the measure concentrates near the equator.  Here, the slices are centered at $0 $ which is outside of  $   B_{\beta \, R}(\bar{y})$ since $\beta < 1$.

\vskip1mm
A term is $O(N^{-1})$ if it is bounded by a constant times $N^{-1}$; this constant may depend on $\beta$ and $\tau$, but not on $N$ or $x$. These bounds will be integrated in $x$, so we need the  $x$ dependence, but the $\beta$ dependence for an $O(N^{-1})$ term
disappears as $N \to \infty$.

\begin{proof}[Proof of Lemma \ref{l:setitup}]
Set $\rho = \beta \, R$.
The solid elliptic mean value property   on $\RR^m$ gives 
\begin{align}	\label{e:gottacomeback}
	 v(\bar{z}) &= \frac{m}{\omega_{m-1}} \, \rho^{-m} \, \int_{B_{\rho}(\bar{z})} v - \frac{1}{\omega_{m-1}} \, 
	\int_0^{\rho} 
	\left\{ (s^{1-m} - s \, \rho^{-m}) \,
	\int_{B_s (\bar{z})} \Delta v \, \right\} ds \, ,
\end{align}
where the last term would vanish if $v$ was actually harmonic. 
 Lemma \ref{l:expandDeltav} gives on $B_{\rho} (\bar{z})$ that 
%\begin{align}
	$|\Delta v| \leq \frac{c}{m} \,  (2\, N)^{ - \frac{n}{2}} \, |y|^{2-N}$
%\end{align}
 where $c$ depends on $\tau$ (and $\sup |u_t| + |u_{tt}|$), but not on $N$.  Thus,  the mean value property for the harmonic function $|y|^{2-N}$ on $\RR^m$ for $r\leq \rho$ gives
\begin{align}
	\int_{B_s(\bar{z})} |\Delta v| \leq \frac{c}{m}  \,  (2\, N)^{ - \frac{n}{2}} \, \int_{B_s(\bar{z})} |y|^{2-N}
	 = \frac{c}{m^2}   (2\, N)^{ - \frac{n}{2}} \,  \, s^m \, \omega_{m-1} \,  R^{2-N} 
	 \, .
\end{align}
Using this in \eqr{e:gottacomeback}  gives 
  that
\begin{align}
	 R^{m-2} \, v(\bar{z}) &= O(N^{-1})   +
	 \frac{m}{\omega_{m-1}} \, (\beta \, R)^{-m}\, R^{m-2} \, \int_{B_{\beta \, R}(\bar{z})} v
	 \, ,
\end{align}
so we see that \eqr{e:IFWT}
gives  \eqr{e:elliptLS} and, thus, the log Sobolev (by Proposition \ref{p:equiv}).
\end{proof}
 
  We will need some notation:
  \begin{itemize}
\item Let $\gamma(\theta)= \gamma (N , \theta)$ be the fraction of the volume of $\SS^{N-1}$ in a geodesic ball of radius $\theta$.
  Clearly, $\gamma(0) = 0$, $\gamma (\pi/2) = \frac{1}{2}$ and $\gamma(\pi) = 1$. 
 \end{itemize} 

 \begin{figure}[htbp]
\includegraphics[totalheight=.22\textheight, width=.3\textwidth]{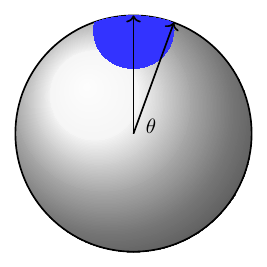}
 \caption*{$\gamma (\theta)=\frac{\Vol ({\color{blue} Region})}{\Vol ({\bf{S}}^{N-1})}$.}
\end{figure}
 
\begin{itemize}
\item  Given $r \in (R-s, R+s)$ with $s^2  =(\beta \, R)^2 - |x|^2 $, 
  the portion of the sphere $\partial B_r (0)$ inside $B_s (\bar{y})$ is a spherical cap. 
  Let 
   $  \theta (x,r)= \theta (\tau ,  \beta , N , x, r)$ be the angle  from the north pole, so that $r \, \theta (x,r)$ is the intrinsic distance from the north pole to the boundary.
\end{itemize}

\vskip4mm

  The next lemma slices the integrand to write it in terms of integrals with respect to the
   ``weight function''   $h (x,r)= r \, \gamma (\theta(x,r))$.

\begin{Lem}	\label{c:slicingit}
We have that
\begin{align}
		 \frac{1}{\omega_{m-1}} \,   \int_{B_{\beta \, R}(\bar{z})} v&=
	  (4\pi)^{ - \frac{n}{2}} \left( 1+ O(N^{-1}) \right)  \,  \int_{B_{\beta \, R} (0)} 
	 \, \int_{R-s}^{R+s} u (x, r^2/(2\, N))  \, h(x,r) \, dr
	  \, 
	 dx   \,  .
\end{align}
\end{Lem}

\begin{proof}
Set $\rho = \beta \, R$.
Since $v= |y|^{2-N} \, (2\, N)^{ - \frac{n}{2}} \, u (x, |y|^2/2N)$, 
Fubini's theorem gives that 
\begin{align}
		 \int_{B_{\rho}(\bar{z})} v&=
	  (2\, N)^{- \frac{n}{2}} \, \int_{B_{\rho} (0)} \int_{B_s (\bar{y})} |y|^{2-N} \, u (x , |y|^2/(2\, N)) \, dy\, dx   \,  .
\end{align}
The definitions of $\gamma(\theta)$ and $\theta (r)$ give that
 \begin{align}	\label{t:slicer1}
	  \int_{B_s (\bar{y})} |y|^{2-N} \, u (x , |y|^2/(2\, N)) \, dy &= \int_{R-s}^{R+s} r^{2-N} \, u (x, r^2/(2\, N)) \,  \omega_{N-1} \, r^{N-1} \, \gamma(\theta(r)) \, dr \notag \\
	  &=  \omega_{N-1} \,  \int_{R-s}^{R+s}  u (x, r^2/(2\, N))  \,  h(x,r) \, dr
	  \, .
\end{align}
The lemma follows    since $N^{\frac{n}{2}}\,\frac{\omega_{m-1}}{\omega_{N-1}} = (2\,\pi)^{\frac{n}{2}} + O(N^{-1})$ by Lemma \ref{l:gamma}. 
\end{proof}

\subsection{Estimates on the weight function}

Define the total mass $\mu (  x)= \mu (\tau,N,\beta,x)$ by
\begin{align}
	  \mu ( x) = \int_{R-s}^{R+s} h ( x,r) \, dr \, .	\label{e:defmu}
\end{align}
The next proposition is the key estimate on $\mu$, giving precise asymptotics for the total mass.  This will rely on sharp asymptotics for the volumes of the slices.

\begin{Pro}	\label{c:mux}
The total mass $\mu(\tau , N,\beta , x)$ is given by 
%$\mu = 2\, \tau \, \e^{ - \frac{R^2-s^2}{4\, \tau}} (1+O(N^{-1}))$.
\begin{align}
	\mu (x) = (1+ O(N^{-1})) \, \frac{\beta^{N-3}}{\sqrt{\pi \, \tau}  } \,    \int_{0}^{s}  \sigma^2 \, \left( 1 - \frac{\sigma^2  + |x|^2}{2\, N \,  \beta^2 \, \tau} \right)^{ \frac{N-3}{2}} \, d\sigma \, .
\end{align}
\end{Pro}

It will  be useful to define $\bar{r} = \bar{r} (N,\beta, \tau, x , r) = \sqrt{R^2 - s^2}$.

\begin{Lem}	\label{l:thetar}
We have that $\cos \theta (x,r) = \frac{r^2 +  \bar{r}^2}{2\, R \, r}$ and 
\begin{align}	\label{e:thetar}
	\theta' (x,r) = - \frac{1}{\sin \theta (x,r)} \,  \left( \frac{r^2  -\bar{r}^2}{2\, R \, r^2} \right) \, .
\end{align}
In particular,  $h> 0$ on $(R-s, R+s)$ and $h=0$ at the end points. Moreover, 
\begin{align}
	\gamma' (\theta(x,r)) \, \theta' (x,r) = - 
\frac{1+ O(N^{-1})}{4\, \sqrt{\pi \tau}\, r^2} \,   \left( r^2 -\bar{r}^2 \right) \, 
	 \left( 1 - \frac{ (r^2 + \bar{r}^2)^2}{8\, N \, \tau \, r^2} \right)^{ \frac{N-3}{2}} \, .
\end{align}
\end{Lem}

\begin{proof}
The angle $\theta = \theta (r)$ satisfies $(r\, \sin \theta , r \, \cos \theta) \in \partial B_s ((0,R)) \subset \RR^2$, so
\begin{align}
	\rho^2 - |x|^2 = s^2 = r^2 \, \sin^2 \theta + (r\, \cos \theta - R)^2 = r^2 + R^2 - 2 \, R \, r \, \cos \theta \, .
\end{align}
Solving for $\cos \theta$ gives the first claim.  Differentiating the first claim in $r$ gives
\begin{align}
	- \sin (\theta) \, \theta_r = \left( \frac{ r^2 +   R^2 - s^2}{2\, R \, r} \right)_r = \frac{ 2\, r }{2\, R \, r} - \frac{ r^2 +  R^2 - s^2}{2\, R \, r^2} = \frac{r^2 -  R^2 + s^2}{2\, R \, r^2} \, . \notag
\end{align}
The next claim is geometrically clear from the triangle inequality since $r=R-s$ is the first time that $\partial B_r$ intersects{\footnote{There is a small subtlety if $s=R$, where the vanishing of $h$ at $r=0$ comes from the factor $r$ in $h$.}} $B_{s} (\bar{y})$ 
and they continue to intersect up until $r = R+s$.  

Define $c = c(N) =\frac{\omega_{N-1}}{\omega_{N-2}}$. Since  
	$\omega_{N-1} = \int_{0}^{\pi} \omega_{N-2} \, \sin^{N-2}(t) \, dt$, we see that 
  $\gamma(\theta) = \frac{1}{c} \, \int_0^{\theta} \sin^{N-2} (t) \, dt$.
%\begin{align}	\label{e:heremygamma}
%	 \, .
%\end{align}
Lemma \ref{l:gamma}
	gives  that
 \begin{align}	
 	c= \frac{\omega_{N-1}}{\omega_{N-2}} =  \left( \frac{2\pi}{N} \right)^{ \frac{1}{2}} (1 + O(N^{-1}))
	\, . \label{e:fromc}
\end{align}
Finally, the first two claims and \eqr{e:fromc} give
\begin{align}
	\gamma'(\theta(x,r)) \, \theta'(x,r) &=- \frac{1}{c\, R} \, \sin^{N-3} (\theta(x,r)) \,  \left( \frac{r^2  -R^2 + s^2}{2\, r^2} \right) \notag \\
	 &=- \frac{1+ O(N^{-1})}{2\, \sqrt{\pi \tau}} \,   \left( \frac{r^2 -R^2 + s^2}{2\, r^2} \right) \, 
	 \left( 1 - \frac{ (r^2 + R^2 -s^2)^2}{8\, N \, \tau \, r^2} \right)^{ \frac{N-3}{2}} \, .
\end{align}
\end{proof}

\begin{proof}[Proof of Proposition \ref{c:mux}]
Within this proof, we will write $\approx$ for equalities up to multiplication by $(1+ O(N^{-1}))$.
If we define $	g= \frac{1}{2} \, r \, h $, then $g(R-s) = g(R+s) = 0$ and $g' = h(r) + \frac{1}{2} \, r^2 \, \gamma'(\theta) \, \theta'$.
  Therefore,  integration by parts and Lemma \ref{l:thetar} give
\begin{align}
	\mu &= \int_{R-s}^{R+s} h(r) \, dr \approx \frac{1}{8\, \sqrt{\pi \, \tau}}  \,    \int_{R-s}^{R+s} \left( r^2 - \bar{r}^2 \right) \,
	\left( 1 -  \frac{(r^2 + \bar{r}^2)^2}{8\, N \, \tau \, r^2} 
	\right)^{ \frac{N-3}{2}}\, dr  \, .
\end{align}
To evaluate this, we will use the change of variables  $t= \frac{\bar{r}^2}{r} + r$, so that 
\begin{enumerate}
\item $t$ goes monotonically from $2\, R$ to $2 \, \bar{r}$ as $r$ goes from $R-s$ to $\bar{r}$ and then back to $2\, R$ as $r \to R+s$. The inverse maps $r_{\pm}(t)$ satisfy  $r^2_{\pm}  (t) \equiv \frac{t^2}{2} - \bar{r}^2 \pm t \, \sqrt{ \frac{t^2}{4} - \bar{r}^2}$. 
\item The differentials of $r$ and $t$ are related by $(r^2 - \bar{r}^2) \, dr = r^2 \,  dt$.
\end{enumerate}
Using this change of variables (and taking advantage of cancellation), we have
\begin{align}
	8\, \sqrt{\pi \, \tau} \, \mu &\approx  \int_{2\, \bar{r}}^{2\, R} \left( r_+^2 - r_-^2 \right) \, 
	\left(1 - \frac{t^2}{8\, N \, \tau} 
	\right)^{ \frac{N-3}{2}} \, dt  \approx  2\, \int_{2\, \bar{r}}^{2\, R}  t \, \sqrt{ \frac{t^2}{4}  - \bar{r}^2} \, \left(1 - \frac{t^2}{8\, N \, \tau} 
	\right)^{ \frac{N-3}{2}} \, dt \, . \notag
\end{align}
Making another change of variable $\sigma = \sqrt{ \frac{t^2}{4}  -\bar{r}^2}$, so that $t\, dt = 4 \, \sigma \, d\sigma$, this becomes
\begin{align}
	\sqrt{\pi \, \tau} \, \mu &\approx    \int_{0}^{s}  \sigma^2 \, \left( 1 - \frac{\sigma^2 +\bar{r}^2}{2\, N \,  \tau} \right)^{ \frac{N-3}{2}} \, d\sigma \, .
\end{align}
This gives the claim since $ 1 - \frac{\sigma^2 +\bar{r}^2}{2\, N \,  \tau} = \beta^2 - \frac{\sigma^2 +|x|^2}{2\, N \,  \tau}$.
\end{proof}

It will be useful to have the following formula for weighted exponential integrals
\begin{align}	\label{e:wtdexp}
	\left( 4 \, \pi \, \tau_0 \right)^{ \frac{k+1}{2}} = \int_{\RR^{k+1}} \e^{ - \frac{|y|^2}{4 \, \tau_0}} \, dy = \omega_k\, \int_0^{\infty} \sigma^k \, \e^{ - \frac{\sigma^2}{4\, \tau_0}} \, d\sigma \, .
\end{align}

 \subsection{Concentration}
 
 The measure $h(x,r)$ is spread out from $r=R-s$ to $r=R+s$ and has total mass $\mu (x)$.
 The  next lemma says that the  measure $h(x,r)$  concentrates around $r=\bar{r}$ when we integrate against $u(x, r^2/2N)$.

\begin{Lem}	\label{l:Urx}
We have that
\begin{align}	 
	\lim_{\beta \to 1} \, \left\{  \lim_{N\to \infty} \, \left|
	 \beta^{-m} \, \int_{B_{\beta \, R} (0)} \left\{ 
	  \mu (x) \, u(x,   \bar{r}^2/2N)
	 -
	 \, \int_{R-s}^{R+s}  u (x, r^2/2N)  \, h(x,r) \, dr
	 \right\}
	  \, 
	 dx  
	   \right| \right\} = 0	 \, . \notag
\end{align}
\end{Lem}

\begin{proof}
Since $|u_t| \leq C$, it follows that 
$ \left| u(x, r^2/2N) - 
	   u(x, \bar{r}^2/2N) \right| \leq  \frac{C}{N} \,   \left| r^2-  \bar{r}^2  \right| $  and, thus, 
\begin{align}	\label{e:majorstep}
	   \int_{R-s}^{R+s}   h(x,r) \, \left| u(x, r^2/2N) - 
	   u(x,  \bar{r}^2/2N) \right| \, dr  & \leq  \frac{C}{N} \,  \int_{R-s}^{R+s} |r^2-\bar{r}^2|   \, h(x,r) \, dr   \, . \notag
\end{align}
We will divide  into the sub-intervals $[R-s, \bar{r}]$ and $[\bar{r} , R+s]$ and define
\begin{align}
	I^+ (N,\beta, x)\equiv \frac{\beta^{-m}}{N} \,  \int_{\bar{r}}^{R+s} |r^2-\bar{r}^2|  \, h(x,r) \, dr  
\end{align}
and similarly for $I^-(N, \beta, x)$.  
Define a function $G(r) =\left( \frac{r^4}{4} - \frac{r^2}{2} \, \bar{r}^2+ \frac{\bar{r}^4}{4} \right) \, \gamma(\theta )$, 
so $G$ vanishes at $R-s$, $\bar{r}$ and $R+s$.  Furthermore, 
\begin{align}
	G' (r) = \left( \frac{r^4}{4} - \frac{r^2}{2} \, \bar{r}^2 + \frac{\bar{r}^4}{4} \right) \, \gamma'(\theta )\, \theta'+
	 (r^2-\bar{r}^2 )    \,  h(r)  \, , 
\end{align}
where the last term is the integrand for $I^+$.  Below, we will use $C$ to denote a constant that may depend on $n, \tau$, but is independent of $x$, $N$ and $\beta$ (for $\beta \in (1/2, 1)$.
Therefore, integration by parts  and Lemma \ref{l:thetar} give that 
\begin{align}
	I^+ &\leq \frac{C}{N} \,  \beta^{-m} \,  \int_{\bar{r}}^{R+s}  (r^2 - \bar{r}^2)^2 \,  \left( 1 - \frac{ (r + \bar{r}^2/r)^2}{8\, N \, \tau} \right)^{ \frac{N-3}{2}}
	\, dr \notag \\
%	 \, .
%\end{align}
%Note that
%\begin{align}
%	  \frac{( r+ \bar{r}^2/r)^2}{8\, N \, \tau}  = \frac{( r- \bar{r}^2/r)^2}{8\, N \, \tau} +  \frac{\bar{r}^2}{2\, N \, \tau} =
%	   \frac{( r- \bar{r}^2/r)^2+ 4\, |x|^2}{8\, N \, \tau} + (1-\beta^2) \, .
%\end{align}
%Using this, we see that
%\begin{align}
%	I^+ &\leq 
&=   \frac{C}{N} \,  \int_{\bar{r}}^{R+s}  (r^2 - \bar{r}^2)^2 \,  \left( 1 - \frac{( r- \bar{r}^2/r)^2+ 4\, |x|^2 }{8\, N\, \beta^2 \, \tau} \right)^{ \frac{N-3}{2}}
	\, dr 
	 \, ,
\end{align}
where the second equality also used that $ \frac{( r+ \bar{r}^2/r)^2}{8\, N \, \tau}   =   \frac{( r- \bar{r}^2/r)^2+ 4\, |x|^2}{8\, N \, \tau} + (1-\beta^2) $.
Using a change of variables $\sigma = r - \bar{r}^2/r $, so that $d\sigma = (1+\bar{r}^2/r^2) \, dr$, we get
\begin{align} 
	 I^+ &\leq\frac{C}{N} 
	\, \int_{0}^{2s}  r^2 \, \sigma^2 \,  \left(1 - \frac{\sigma^2+ 4\, |x|^2 }{8\, N \, \beta^2 \, \tau} \right)^{ \frac{N-3}{2}}
	\, \frac{d\sigma}{1+\bar{r}^2/r^2}  \leq  \frac{C}{N} \, \e^{   - \frac{ |x|^2 }{8 \, \beta^2 \, \tau} } \, \int_{0}^{2s}   (\bar{r}^2 + \sigma^2)  \, \sigma^2 \,  \e^{   - \frac{\sigma^2 }{32 \, \beta^2 \, \tau} }	\, d\sigma
	 \, , \notag
\end{align}
where the last inequality used   Lemma \ref{c:taylor} (and $N > 6$).
Using   \eqr{e:wtdexp} to bound the $d\, \sigma$ integrals gives
\begin{align} 
	 I^+ &\leq\frac{C}{N}   \, \e^{   - \frac{ |x|^2 }{8 \, \beta^2 \, \tau} } \,   (\bar{r}^2 + 1)   \leq 	
	 \frac{C}{N}   \, \e^{   - \frac{ |x|^2 }{8 \, \beta^2 \, \tau} } \,  \left( 1 + |x|^2 + (1-\beta)^2\, R^2 \right)  \, .
\end{align}
Integrating in $x$ and taking the limit as $N \to \infty$ gives that
\begin{align}
	\lim_{N \to \infty} \, \int I^+ (N,\beta , x) \, dx \leq C \, (1-\beta^2) \, . 
\end{align}
This goes to zero as $\beta \to 1$.  The estimate for $I^-$ follows similarly, completing the proof.
\end{proof}

\begin{Lem}	\label{l:asympto}
We have that
\begin{align}	\label{e:IbetaN}
	\lim_{N\to \infty} \left\{   \frac{\beta^{-m}}{2\, \tau} \, \int_{B_{\rho} (0)} 
	 \, \mu (x) \, u(x,  \frac{\bar{r}^2}{2N}) \, dx \right\} = \beta^{-n} \, \int_{\RR^n} 
	\e^{ - \frac{|x|^2}{4\, \beta^2 \, \tau}} \, u(x, (1-\beta^2) \,  \tau) \, dx 
	 \, .
\end{align}
\end{Lem}

\begin{proof}
Set $\bar{\tau} = (1-\beta^2) \, \tau$.
We use $\approx$ for terms that are equal up to a factor of $(1+O(N^{-1}))$. 
Set  $H(x,N) = \int_{0}^{s}  \sigma^2 \, \left( 1 - \frac{\sigma^2 +|x|^2}{2\, N \, \beta^2 \, \tau} \right)^{ \frac{N-3}{2}} \, d\sigma$.
Let $I(N)$ denote the term in brackets on the left-hand side of \eqr{e:IbetaN} and use
Proposition \ref{c:mux} to get
\begin{align}
	I (N) =  \frac{\beta^{-n-3}}{2\, \tau\, \sqrt{\pi \, \tau}} \, \int_{B_{\rho} (0)} 	 u(x,  \bar{\tau} + \frac{|x|^2}{2N}) \, H(x,N)
	 \, dx \, .
\end{align}
Observe that for each fixed $x$ and $\beta$, we have that   $ u(x,  \bar{\tau} + \frac{|x|^2}{2N})  \to u(x, \bar{\tau})$ and 
\begin{align}
	\lim_{N\to \infty} H(x,N) &=  \int_{0}^{\infty}  \sigma^2 \, \e^{ - 
	  \frac{\sigma^2 +|x|^2}{4 \, \beta^2 \, \tau} } \, d\sigma = \e^{ - \frac{|x|^2}{4\, \beta^2 \, \tau}} \, 
	    \int_{0}^{\infty}  \sigma^2 \, \e^{ - 
	  \frac{\sigma^2}{4 \, \beta^2 \, \tau} } \, d\sigma = \e^{ - \frac{|x|^2}{4\, \beta^2 \, \tau}} \, \left( 2 \, \sqrt{\pi} \, \beta^3 \, \tau^{ \frac{3}{2}} \right)
	  \, , \notag
\end{align}
where the last equality used  \eqr{e:wtdexp}.
Thus, once we show that
$u(x, \bar{\tau} + \frac{|x|^2}{2N}) \, H(x,N)$ is dominated, 
 the dominated convergence theorem will give
\eqr{e:IbetaN}.

  Since $N>6$ is large,  Lemma \ref{c:taylor}
gives that
\begin{align}
	 \left( 1 - \frac{\sigma^2 +|x|^2}{2\, N \, \beta^2 \, \tau} \right)^{ \frac{N-3}{2}} \leq 	 \left( 1 - \frac{\sigma^2 +|x|^2}{2\, N \, \beta^2 \, \tau} \right)^{ \frac{N}{4}} \leq
	 \e^{ -  \frac{\sigma^2 +|x|^2}{8 \, \beta^2 \, \tau}} =  \e^{ -  \frac{|x|^2}{8 \, \beta^2 \, \tau}} \, \left( \e^{ -  \frac{\sigma^2}{8 \, \beta^2 \, \tau}} \right) \, .
\end{align}
Therefore, we have for every $N$ that
\begin{align}
	H(x,N) \leq   \e^{ -  \frac{|x|^2}{8 \, \beta^2 \, \tau}} \, \int_0^{\infty} \sigma^2 \,  \left( \e^{ -  \frac{\sigma^2}{8 \, \beta^2 \, \tau}} \right) \, d\sigma =  \e^{ -  \frac{|x|^2}{8 \, \beta^2 \, \tau}} \, \left( 2^{ \frac{7}{2}} \, \sqrt{\pi} \, \beta^3 \, \tau^{ \frac{3}{2}} 
	\right) \, , 
\end{align}
which is integrable in $x$, completing the proof.
\end{proof}

\begin{proof}[Proof of Theorem \ref{p:repnf}]
Set $\bar{\tau} = (1-\beta^2) \, \tau$ and $\rho = \beta \, R$.
Lemmas   \ref{c:slicingit} and \ref{l:Urx} give that
\begin{align}
	\lim_{\beta \to 1} \left\{\lim_{N \to \infty} \,    \frac{\beta^{-m} }{2\, \tau \, \omega_{m-1}}\, \int_{B_{\rho}(\bar{z})} v  \right\} &=
	\lim_{\beta \to 1} \left\{\lim_{N \to \infty} \,    \frac{\beta^{-m} }{2\, \tau\,  (4\pi)^{ \frac{n}{2}}   } \,  \int_{B_{\rho} (0)} 
	 \, \int_{R-s}^{R+s} u (x, r^2/(2\, N))  \, h(x,r) \, dr
	  \, 
	 dx \right\} \notag \\
	 &= \lim_{\beta \to 1} \left\{ \lim_{N \to \infty} \,    \frac{\beta^{-m} }{2\, \tau\,  (4\pi)^{ \frac{n}{2}}   } \, 
	 \int_{B_{\rho} (0)} 
	  \mu (x) \, u(x, \bar{\tau} + \frac{|x|^2}{2N}) \, dx
	\right\} \\
	&=    (4\pi)^{ - \frac{n}{2}}\,  
 \lim_{\beta \to 1} \left\{	 \beta^{-n} \, \int_{\RR^n} 
	\e^{ - \frac{|x|^2}{4\, \beta^2 \, \tau}} \, u(x, \bar{ \tau}) \, dx 
	\right\} 
\notag \, , 
\end{align}
where the last equality used Lemma \ref{l:asympto}.
The dominated convergence theorem gives \eqr{e:IFWT} and, thus, Lemma \ref{l:setitup}
 gives the claim.
\end{proof}

\appendix
\section{}

 \begin{Pro}	\label{p:equiv}
 The log Sobolev is a consequence of elliptic monotonicity and  \eqr{e:elliptLS}.
 \end{Pro}
 
 \begin{proof}
  Recall that the modified Shannon 
 entropy  is given by
\begin{align}
	\tilde{\cS} &= - \frac{n}{2} \, \log t  - \int u \, \log u \, .
\end{align}
 The first step is to show that the log Sobolev follows from 
 \begin{align}	\label{e:LSA2}
	\limsup_{t\to \infty} \tilde{\cS} \geq \frac{n}{2} + \frac{n}{2} \, \log (4\, \pi) \, .
\end{align}
The $\cW$ functional can be written as a   derivative by
 $
	\cW  =   (t\, \tilde{\cS})' + \frac{n}{2}$.
	Monotonicity of $\cW$, \cite{P},  was shown{\footnote{\cite{BR} assumed  $M$ was   closed;
	 monotonicity follows on $\RR^n$ (and more general spaces) from \eqr{e:key1}, \eqr{e:key2}.}} 
	 was shown to be a high-dimensional limit of \cite{C} 
in Theorem $4.10$ in \cite{BR}.
Fix some $t_0 > 0$. 
Monotonicity implies that $\cW (t) \leq \cW (t_0)$ for all $t \geq t_0$. Therefore, 
since $(t \, \tilde{\cS})' = \cW - \frac{n}{2}$, we get for 
 $T > t_0$  that
\begin{align}
	T \, \tilde{\cS}(T) \leq t_0 \,  \tilde{\cS}(t_0) + \int_{t_0}^T \left( \cW (t_0) - \frac{n}{2} \right) 
	\leq t_0 \,  \tilde{\cS}(t_0) +(T-t_0) \,  \left( \cW (t_0) - \frac{n}{2} \right) \, .
\end{align}
Combining this with the lower bound \eqr{e:LSA2} gives that 
\begin{align}
	\frac{n}{2} + \frac{n}{2} \, \log (4\, \pi) \leq \limsup_{t \to \infty}\, \tilde{\cS} \leq  \cW (t_0) - \frac{n}{2} \, ,
\end{align}
so   $\cW \geq \frac{n}{2} \, (2 + \log (4\pi))$, which is well-known to be equivalent to the log Sobolev inequality.

It remains to show that \eqr{e:elliptLS} implies \eqr{e:LSA2}.
Set $\tilde{u} (x,t) = t^{ \frac{n}{2}} \, u(\sqrt{t} \, x , t)$.  Note that the integral of $\tilde{u} $ must be one by using the change of variables formula (since the integral of $u$ is one).
Using the change of variables $x= \sqrt{t} \, \tilde{x}$, we see that 
\begin{align}
	\cS (t) &= - \int u(x,t) \, \log u (x,t) \, dx  =  - \int  \tilde{u}( \tilde{x},t) \, \log  [ t^{ -\frac{n}{2} } \tilde{u} ( \tilde{x},t)] \, d\tilde{x} \\
	&= \frac{n}{2} \, \log t \, \int \tilde{u}  - \int  \tilde{u} \, \log  \tilde{u} =  \frac{n}{2} \, \log t   - \int  \tilde{u} \, \log  \tilde{u} 
	\, . \notag 
\end{align}
We can apply \eqr{e:elliptLS}  to get that
 \begin{align}
	\tau^{ \frac{n}{2}} \, u (0,\tau) 
	=   \lim_{N \to \infty} \, R^{m-2} \, v(\bar{z}) 
	=  \left( 4 \, \pi \right)^{ - \frac{n}{2}}  \, \int u(x,0) \, \e^{ - \frac{|x|^2}{4 \, \tau}} \, dx 
	 \, .
\end{align}
Applying this to a translation of $u$ (which still satisfies the heat equation) gives that
 \begin{align}
	\left( 4 \, \pi \right)^{  \frac{n}{2}}   \,  \tilde{u} (x, \tau) = \left( 4 \, \pi \right)^{  \frac{n}{2}}   \, \tau^{ \frac{n}{2}} \, u (\sqrt{\tau} \, x,\tau) 
		=  \int u(\tilde{x}  ,0) \, \e^{ - \frac{|\tilde{x}-\sqrt{\tau} \, x|^2}{4 \, \tau}} \, d\tilde{x}
		= \e^{- \frac{|x|^2}{4}} \,  \int u(\tilde{x}  ,0) \, \e^{ -  \frac{|\tilde{x}|^2}{4\tau} +   \frac{\langle \tilde{x} , x \rangle}{2\sqrt{\tau}} } \, d\tilde{x}
	 \, . \notag
\end{align}
Thus, we see that
  $\tilde{u}$ goes to the standard gaussian $(4\, \pi)^{ - \frac{n}{2}} \, \e^{ - \frac{|x|^2}{4}}$, so the dominated convergence theorem gives that
\begin{align}
	\lim_{t \to \infty} \tilde{\cS}(t) = (4\, \pi)^{ - \frac{n}{2}} \, \int  \e^{ - \frac{|x|^2}{4}} \left( \frac{n}{2} \, \log (4\, \pi) + \frac{|x|^2}{4} \right) =
	\frac{n}{2} + \frac{n}{2} \, \log (4\, \pi) \, ,
\end{align}
 completing the proof.
 \end{proof}

\subsection{Volumes of spheres}

Recall that the volume of the $k$-dimensional unit sphere in $\RR^{k+1}$ is
$
\omega_{k}=\frac{2\,\pi^{\frac{k+1}{2}}}{\Gamma (\frac{k+1}{2})}$, 
where the Gamma function $\Gamma$ is given on positive integers by
$\Gamma (k)=(k-1)!$.

\begin{Lem}    \label{l:gamma}
If $m = n +N$ and $N$ is large, then 
\begin{align}
	N^{\frac{n}{2}}\, \frac{\omega_{m-1}}{\omega_{N-1}} & = (2\,\pi)^{\frac{n}{2}} + O(N^{-1}) 
	\,  .
\end{align}
\end{Lem}

\begin{proof}
Stirling's formula gives that $\Gamma (k) = \sqrt{2\pi/k} \, (k/\e)^k \, (1+ O(k^{-1}))$. Thus, we have
\begin{align}
 \frac{\omega_{m-1}}{\omega_{N-1}} &=\frac{\pi^{\frac{n}{2}}\,\Gamma (\frac{N}{2})}{\Gamma(\frac{N+n}{2})}=
 \frac{\pi^{\frac{n}{2}}\,   \sqrt{N+n} \, (N/2\e)^{N/2} }{  \sqrt{N} \, ((N+n)/2\e)^{(N+n)/2}     }  \, (1+ O(N^{-1}))  
 = (1+O(N^{-1})) \, \left( \frac{2\, \pi}{N} \right)^{ \frac{n}{2}}
 \,  , \notag
\end{align}
where the last equality also used that $(1+n/N)^{-N/2} = \e^{- n/2} + O(N^{-1})$.
% To prove this, note that
%\begin{align}
% \frac{\omega_{m-1}}{\omega_{N-1}} &=\frac{\pi^{\frac{n}{2}}\,\Gamma (\frac{N}{2})}{\Gamma(\frac{N+n}{2})}=\pi^{\frac{n}{2}}\,\left(\frac{N}{2}\right)^{-\frac{n}{2}}\frac{\left(\frac{N}{2}\right)^{\frac{n}{2}}\,\Gamma (\frac{N}{2})}{\Gamma (\frac{N+n}{2})}
 %= (1+O(N^{-1})) \, \left( \frac{2\, \pi}{N} \right)^{ \frac{n}{2}} \, .
  %\frac{\left(\frac{N}{2}\right)^{\frac{n}{2}}\,\Gamma (\frac{N}{2})}{\Gamma (\frac{N+n}{2})} &=1 + O(N^{-1}) \,  .
%\end{align}
\end{proof}

\begin{Lem}   \label{c:taylor}
If   $2\leq N$ and $0 \leq \delta < N$, then
%\begin{align}
$\left(1-\frac{\delta}{N}\right)^{\frac{N}{2}}
\leq \e^{-\frac{\delta}{2}}   \, .$
%\end{align}
\end{Lem}

\begin{proof}
Define $h(y) = y + \log (1-y) $ for $y \in [0,1)$ and observe that $h(0) = 0$ and
 \begin{align}
 	h'(y) = 1 - \frac{1}{1-y} = \frac{-y}{1-y} \leq 0 \, , 
 \end{align}
 so we have that $h\leq 0$ on $[0,1)$.  Setting $y= \frac{\delta}{N}$ gives that
 \begin{align}
 	\frac{\delta}{N} + \log (1 - \delta/N) \leq 0 \, , 
 \end{align}
 which gives that $\frac{N}{2} \, \log (1- \delta/N) \leq - \frac{\delta}{2}$.  Exponentiating this gives the   claim.
\end{proof}

 As a consequence, as observed by Poincar\'e, the projection of the uniform probability measure on $\SS_{\sqrt{2N}}^{m-1}$ to $\RR^n$ converges to the standard gaussian measure.  This follows since
\begin{align}
	\frac{ (2N)^{ - \frac{m-1}{2}} }{\omega_{m-1}} \, \int_{ \SS_{\sqrt{2N}}^{m-1} } f(x) &= 
	\frac{ (2N)^{ - \frac{m-1}{2}} }{\omega_{m-1}} \, \int_{  \{   |x|^2 \leq 2\, N\} } f(x) \, \Vol \left( \{   |y|^2 = 2\, N - |x|^2 \} \right) \notag \\
	&= \frac{ (2N)^{ - \frac{m-1}{2}}\, \omega_{N-1} }{\omega_{m-1}} \, \int_{  \{   |x|^2 \leq 2\, N\} } f(x) \, \left(   2\, N - |x|^2  \right)^{ \frac{N-1}{2}} \\
	&= \frac{ (2N)^{ - \frac{n}{2}}\, \omega_{N-1} }{\omega_{m-1}} \, \int_{  \{   |x|^2 \leq 2\, N\} } f(x) \, \left(  1 - \frac{|x|^2}{2\, N}  \right)^{ \frac{N-1}{2}} 
	\, . \notag
\end{align}
Finally, Lemmas \ref{l:gamma} and  \ref{c:taylor}  give that
\begin{align}
	 \frac{ (2N)^{ - \frac{n}{2}}\, \omega_{N-1} }{\omega_{m-1}}  \to (4\, \pi)^{ - \frac{n}{2}}  {\text{ and }}
 \left(  1 - \frac{|x|^2}{2\, N}  \right)^{ \frac{N-1}{2}}  \to \e^{ - \frac{|x|^2}{4}} \, .	\notag
\end{align}


\begin{thebibliography}{A}
 
  \bibitem[AFM]{AFM}
 V. Agostiniani, M. Fogagnolo, and L. Mazzieri, {\it{Sharp geometric inequalities for closed hypersurfaces in manifolds with nonnegative Ricci curvature}}, Invent. Math. 222 (2020), no. 3, 1033--1101.
 
%  \bibitem[AFM2]{AFM2}
% V. Agostiniani, M. Fogagnolo, and L. Mazzieri, {\it{Minkowski inequalities via nonlinear potential theory}},
%  Arch. Ration. Mech. Anal. 244 (2022), no. 1, 51--85.
  
  
  \bibitem[AMMO]{AMMO}  
   V. Agostiniani, C.  Mantegazza, L.  Mazzieri and F.  Oronzio,
   {\it{Riemannian Penrose inequality via Nonlinear
 Potential Theory}}, arXiv:2205.11642

 
  \bibitem[AMO]{AMO}  
 V. Agostiniani, L. Mazzieri and F. Oronzio, \emph{A Green's function proof of the positive mass theorem},
  Comm. Math. Phys. 405 (2024), no. 2, Paper No. 54, 23 pp.

  
 \bibitem[Bj]{Bj}
J. Bennett, \emph{Heat-flow monotonicity related to some inequalities in euclidean analysis}, 
Proceedings of the 8th International Conference on Harmonic Analysis and Partial Differential Equations,
 El Escorial, Spain. Contemporary Mathematics, 505 (2010), 85--96.
 
 \bibitem[Bo]{Bo}
C. Borell, \emph{The Brunn-Minkowski inequality in Gauss space}. 
 Invent. Math. 30 (1975), no. 2, 207--216.
 
\bibitem[BR]{BR}
 I. Bustamante and M. Reiris,
 {\it{Deriving Perelman's Entropy from Colding's Monotonic Volume}}, 
 https://arxiv.org/2501.12949, Journal f\"ur die reine und angewandte Mathematik, to appear.
 
%\bibitem[ChC]{ChC}  
%J. Cheeger and T.H. Colding, 
%\textit{Lower bounds on Ricci curvature and the almost rigidity of warped products}. 
%Ann. of Math. (2) 144 (1996), no. 1, 189--237.

%\bibitem[CgY]{CgY} 
%S.Y. Cheng and S.T. Yau, \emph{Differential equations on Riemannian manifolds and their geometric applications}. Comm. Pure Appl. Math. 28 (1975), no. 3, 333--354.

\bibitem[C]{C}
T.H. Colding, 
\textit{New monotonicity formulas for Ricci curvature and applications}. I. Acta Math. 209 (2012), no. 2, 229--263.

\bibitem[CM1]{CM1}
T.H. Colding and W.P. Minicozzi II, 
\textit{Monotonicity and its analytic and geometric implications}. Proc. Natl. Acad. Sci. USA 110 (2013), no. 48, 19233--19236.

\bibitem[CM2]{CM2}
T.H. Colding and W.P. Minicozzi II, 
\textit{Ricci curvature and monotonicity for harmonic functions}. Calc. Var. Partial Differential Equations 49 (2014), no. 3-4, 1045--1059.
 
 \bibitem[D]{D}
 B. Davey,  
 \textit{Parabolic theory as a high-dimensional limit of elliptic theory}. Arch
Rational Mech Anal., 228, (2018) 159-196.

\bibitem[Fe]{Fe}
P. 
Federbush, 
Partially alternate derivation of a result of Nelson. 
J. Math. Phys. 10, 50-52 (1969).

\bibitem[Gr]{Gr}
L. Gross,  
{\it Logarithmic Sobolev inequalities.}
Amer. J. Math. 97 (1975), no. 4, 1061--1083. 

%\bibitem[GS]{GS}
%N. Guillen and L. Silvestre, 
%\emph{The Landau equation does not blow up}, Acta Math., to appear.  

%\bibitem[ISV]{ISV} C. Imbert, L. Silvestre and C. Villani, 
%\emph{On the monotonicity of the Fisher information for the Boltzmann equation},  https://arxiv.org/abs/2409.01183


\bibitem[Le1]{Le1}
M.  Ledoux,  
\emph{Concentration of measure and logarithmic Sobolev inequalities}. Seminaire de Probabilites, XXXIII, 120-216, Lecture Notes in Math., 1709, Springer, Berlin, 1999.

\bibitem[Le2]{Le2}
M.  Ledoux,  
\emph{The concentration of measure phenomenon}. 
Mathematical Surveys and Monographs, 89. American Mathematical Society, Providence, RI, 2001.

\bibitem[Le3]{Le3}
M.  Ledoux, 
\emph{Heat flows, geometric and functional inequalities}, 
Proceedings of the International Congress of Mathematicians -- Seoul 2014. Vol. IV, 117--135, Kyung Moon Sa, Seoul, 2014.

%Ledoux, M. Differentials of entropy and Fisher information along heat flow: a
%brief review of some conjectures. Preprint, available online on the personal page
%https://perso.math.univ-toulouse.fr/ledoux/files/2024/09/Entropy-conjectures.pdf, 2023.

\bibitem[LY]{LY} 
P. Li and S.T. Yau, \emph{On the parabolic kernel of the Schr\"odinger operator}. Acta Math. 156 (1986), no. 3-4, 153--201.

\bibitem[M]{M} 
H.P. McKean, 
\emph{Geometry of differential space}.
Ann. Probability 1 (1973), 197--206.

\bibitem[Mi]{Mi} 
V.D. Milman, 
\emph{The heritage of P. L\'evy in geometrical functional analysis}.
Colloque Paul L\'evy sur les Processus Stochastiques (Palaiseau, 1987).
Ast\'erisque No. 157-158 (1988), 273-301.

\bibitem[N]{N} 
A. Naor, \emph{Concentration of measure}, lecture notes 2008, https://web.math.princeton.edu/~naor/

\bibitem[Ne]{Ne}
E. Nelson,  {\it A quartic interaction in two dimensions}. 1966 Mathematical Theory of Elementary Particles (Proc. Conf., Dedham, Mass., 1965) pp. 69--73 M.I.T. Press, Cambridge, Mass.

%\bibitem[Ni]{Ni}
%L. Ni, \emph{The entropy formula for linear heat equation}. J. Geom. Anal. 14 (2004), no. 1, 87--100

%\bibitem[Ni2]{Ni2} 
%L. Ni, \emph{Addenda to: "The entropy formula for linear heat equation''},  J. Geom. Anal. 14 (2004), no. 2, 369--374.

\bibitem[P]{P}
G. Perelman, \emph{The entropy formula for the Ricci flow and its geometric applications}, https://arxiv.org/abs/math/0211159.

\bibitem[Po]{Po}
H. Poincar\'e, \emph{Calcul des probabilit\'es}. Paris: Gauthier-Villars 1912.

%\bibitem[SaTo]{SaTo}
%G. Savar\'e and G. Toscani, \emph{The concavity of R\'enyi entropy power}. IEEE Trans. Inform. Theory 60 (2014), no. 5, 2687--2693.

%\bibitem[ST]{ST}  
%V.N. Sudakov and B.S. Tsirelson, 
%\emph{Extremal properties of half-spaces for spherically invariant measures}.  
%Problems in the theory of probability distributions, II.
%Zap. Naun. Sem. Leningrad. Otdel. Mat. Inst. Steklov. (LOMI) 41 (1974), 14--24, 165.

\bibitem[Sc]{Sc} E. Schmidt, \emph{Die Brunn-Minkowskische Ungleichung und ihr Spiegelbild sowie die isoperimetrische Eigenschaft der Kugel in der euklidischen und nichteuklidischen Geometrie}. I. Math. Nach.
1, 81-157 (1948).

\bibitem[St]{St} A.J. Stam, 
\emph{Some inequalities satisfied by the quantities of information of Fisher and Shannon}.
Information and Control 2 (1959), 101--112.

 
\bibitem[T]{T}
T. Tao, \emph{Comparison geometry, the high-dimensional limit, and Perelman reduced volume},\\
https://terrytao.wordpress.com/2008/04/27/285g-lecture-9-comparison-geometry-the-high-dimensional-limit-and-perelman-reduced-volume/

 
\bibitem[van]{van} 
R. van Handel, 
\emph{Probability in High Dimension}, 
Lecture notes,\\ https://web.math.princeton.edu/~rvan/APC550.pdf.



 
\bibitem[V]{V} 
C. Villani, \emph{Fisher Information in Kinetic Theory}, https://arxiv.org/abs/2501.00925

%\bibitem[W]{W}  
%N.  Wiener, \emph{Differential-space}. J. Math. and Phys. 2 (1923), 131--174.

 

\end{thebibliography}
\end{document}